\documentclass[12pt,a4paper]{amsart}

\usepackage[latin1]{inputenc}
\usepackage[all]{xy}
\usepackage{amssymb, amsmath, amscd, geometry}
\usepackage{graphicx}
\usepackage{setspace}

\numberwithin{equation}{section}
\input xy
\xyoption{all}

\usepackage{latexsym}
\usepackage[usenames]{color}
\usepackage{epic} 
\usepackage{mathrsfs}
\usepackage{euscript}
\usepackage{rotating}
\usepackage{verbatim}
\usepackage{todonotes}

\newtheorem{lemma}{Lemma}[section]
\newtheorem{theorem}[lemma]{Theorem}
\newtheorem{corollary}[lemma]{Corollary}

\newtheorem{prop}[lemma]{Proposition}

\newtheorem*{theorem*}{Theorem}

\theoremstyle{definition}

\newtheorem{remark}{Remark}[section]

      \newcommand{\R}{{\mathbb R}}

\newcommand{\ra}{\rightarrow}

\newcommand{\sign}{\operatorname{sign}}

\title[Distance between Haar and Gaussian matrices]{Euclidean distance between Haar orthogonal and Gaussian matrices}

\author{C. E. Gonz\'alez-Guill\'en}
\author{C. Palazuelos}
\author{I. Villanueva}

\addtocounter{tocdepth}{-1}

\begin{document}

\addtolength{\parskip}{+1ex}


\maketitle
{\let\thefootnote\relax\footnote{The final publication is available at Springer via http://dx.doi.org/10.1007/s10959-016-0712-6}}
\setcounter{footnote}{0}

\begin{abstract}
In this work we study a version of the general question of how well a Haar distributed orthogonal matrix can be approximated by a random Gaussian matrix. Here, we consider a Gaussian random matrix $Y_n$ of order $n$ and apply to it the Gram-Schmidt orthonormalization procedure by columns to obtain a Haar distributed orthogonal matrix $U_n$. If $F_i^m$ denotes the vector formed by the first $m$-coordinates of the $i$th row of $Y_n-\sqrt{n}U_n$ and $\alpha=\frac{m}{n}$, our main result shows that the Euclidean norm of $F_i^m$ converges exponentially fast to $\sqrt{ \left(2-\frac{4}{3} \frac{(1-(1 -\alpha)^{3/2})}{\alpha}\right)m}$, up to negligible terms. 

To show the extent of this result, we use it to study the convergence of the supremum norm $\epsilon_n(m)=\sup_{1\leq i \leq n, 1\leq j \leq m} |y_{i,j}- \sqrt{n}u_{i,j}|$ and we find a coupling that improves by a factor $\sqrt{2}$ the recently proved best known upper bound of $\epsilon_n(m)$. Our main result has also applications in Quantum Information Theory.
\end{abstract}
\section{Introduction}
One of the classical problems in random matrix theory is to compare a random Gaussian matrix  $Y_n=(y_{i,j})_{i,j=1}^n$ with a Haar distributed random matrix $ U_n=(u_{i,j})_{i,j=1}^n$ in the orthogonal group $\mathcal O(n)$. This problem dates back to the famous work \cite{Borel}, where Borel showed that the distribution of one single coordinate of $U_n$  converges to the distribution of one single coordinate of $Y_n$, when properly normalized. That is, for a fixed pair $(i,j)$ we have that $\sqrt{n} u_{i,j}$ converges in distribution to a standard normal. Since then, many authors have studied the problem of how many entries of $\sqrt{n}U_n$ can be simultaneously well approximated by the corresponding entries of $Y_n$; that is, by independent standard normal distributions. Indeed, a number of papers in the 1980's made further progress in this direction (see for instance \cite{DiFr}, \cite{Ga}, \cite{St}, \cite{Yor}). The reader can find a more complete introduction to this problem in  \cite[Section 6.3]{Diaconis} and the references therein.

Some of the previous works focused on the variation distance. More precisely, in \cite{St}  the authors proved that the joint distribution of the first $l_n$ coordinates of the first column of $Y_n-\sqrt{n}U_n$ converges to 0 in variation distance as $n$ grows to infinity, provided that $l_n=o(\sqrt{n})$. Later, the order was improved to $l_n=o(n)$ (\cite{DiFr}). More recently, in \cite{DiEaLa}, it was proven that the joint distribution of the upper left $l_n\times m_n$ block of $Y_n-\sqrt{n}U_n$ converges to 0 in variation distance provided that $l_n, m_n$ are both $o(n^\frac{1}{3})$.  Later, in \cite{Co}, this order was improved to $O(n^\frac{1}{3})$.

The latest major achievement in this direction came from \cite{Ji, Ji2}. In those papers the author showed that the joint distribution of the upper left $l_n\times m_n$ block of $Y_n-\sqrt{n}U_n$ converges to 0 in variation distance if and only if $l_n,m_n$ are both $o(n^\frac{1}{2})$. This settled the long standing open problem of finding the best ratio in the variation distance case (see \cite[Section 6.3]{Diaconis}). In the same paper Jiang  also showed \cite[Theorem 3]{Ji} the existence of a coupling between $Y_n$ and $U_n$ such that $$\epsilon_n(m)=\sup_{1\leq i \leq n, 1\leq j \leq m} |y_{i,j}- \sqrt{n}u_{i,j}|$$ converges to 0 in probability if  and only if $m=o(\frac{n}{\log n})$. Moreover, if $m=\frac{\beta n}{\log n}$ then the previous supremum converges in probability to $2\sqrt{\beta}$. These results have been applied in \cite{TaoVu} to study the eigenvector distribution of a wide class of Wigner ensembles. For further history and applications of these results, see \cite{Ji, Ji2, TaoVu}. Recent different generalizations of Borel's theorem can be seen in \cite{CoSt, DaDiNe}.

Given the relevance of the Euclidean norm in many contexts, and motivated by these previous works, in this paper we study the behavior of the Euclidean norm of blocks of  $Y_n-\sqrt{n}U_n$. We are interested not only in the order needed for convergence to 0, but in the general value of the norm. To show the extent of our main result, we show later how to recover from it  one of the  main results of \cite{Ji}. We have also applied it to solve a question in Quantum Information Theory \cite{GoJiPaVi}.

Let us define the notation needed to state our main result. Our probability spaces will be $\mathbb R^{n^2}$ with the Gaussian measure.  For every $n\in \mathbb N$, $Y_n=(y_{i,j})_{i,j=1}^n$ will be a Gaussian random matrix; that is, the variables $(y_{i,j})_{i,j=1}^n$ are independent standard normal variables. For every $1\leq j \leq n$ we consider the column vector  ${\bf y}_j=(y_{i,j})_{i=1}^n$. With probability 1, they form a basis of $\mathbb R^n$. Following \cite{Ji}, we apply the Gram-Schmidt orthonormalization procedure to $({\bf y}_1,\ldots, {\bf y}_n)$ to obtain an orthonormal basis $(\boldsymbol{\nu}_j)_{j=1}^n$. We call $U_n$ the matrix $(\nu_{i,j})_{i,j=1}^n$. We recall that $U_n$ is Haar distributed.

For every $1\leq m \leq n$ and for every $1\leq i \leq n$, let $F_i^m$ be the vector formed by the the first $m$-coordinates of the $i$th row of $Y_n-\sqrt{n}U_n$. We describe the asymptotic generic behavior of  $\|F_i^m\|$, where $\|\cdot\|$ is the Euclidean norm. Let $[x]$ denote the integer part of $x$.
\begin{theorem}\label{main}
Let $n\in \mathbb N$, let $0<\alpha\leq 1$ be fixed and let $m= [\alpha n]$. Let $Y_n$, $U_n$, $F_i^m$ be as above. Then, there exists $0< \delta<\frac{1}{2}$ such that,  

$$\sup_i\|F_i^m\| \leq \sqrt{  \left(2- \frac{4}{3} \frac{(1-(1 -\alpha)^{3/2})}{\alpha}\right)m} + O(m^{\delta}) $$

and $$ \inf_i\|F_i^m\| \geq \sqrt{ \left(2- \frac{4}{3} \frac{(1-(1 -\alpha)^{3/2})}{\alpha}\right)m} -O(m^{\delta}),$$
both with probability exponentially close to 1 as $n$ grows to infinity.
\end{theorem}

We have chosen this presentation of the main theorem for the sake of clarity. The actual proof shows further insight into the result. Specifically, we want to mention that there is a trade off between the rate of the concentration and the order $\delta$ appearing in Theorem \ref{main}. In our proof we show how to make $\delta=\frac{2}{5}$  keeping a very fast concentration rate. Nevertheless, the parameters can be changed easily to obtain a different value for $\delta$, at the cost of modifying the rate of the exponential convergence of the probability. 

We clarify next some aspects of our result. The coupling is given by the Gram-Schmidt procedure performed columnwise. It is simple to see that the expectation of the norms of the columns of $Y_n-\sqrt{n}U_n$ increases with the index of the column. Our main result says that the norms of the truncated rows will all concentrate exponentially around the same value. This ``flatness'' phenomenon is very relevant in our applications.   

Therefore, our main contribution can be seen as a ``block-delocalization'' result. We show that the Euclidean norm of the whole block $$\left(\sum_{1\leq i \leq n, 1\leq j \leq m} |y_{i,j}-\sqrt{n}u_{i,j}|^2\right)^\frac{1}{2}$$is well delocalized among the Euclidean norm of the rows. The lack of independence is the main difficulty  in this case and we need to deal with different technicalities to overcome this and prove our result. Our main tools are standard versions of the concentration of measure phenomenon and the Gram-Schmidt procedure. 

Theorem \ref{main} considers the case of constant $\alpha=\frac{m}{n}$. We describe in the following corollary the behaviour of the euclidean norm of the rows in the case $\alpha\rightarrow 0$.

\begin{corollary}\label{correferee}
For every  $n\in \mathbb N$, choose $m\in \mathbb N$ with $0\leq m\leq n$ so that the sequence $\alpha_n=\frac{m}{n}$ verifies  $\lim\limits_{n\rightarrow \infty} \alpha_n=0$. Let $Y_n$, $U_n$, $F_i^{m}$ be as above. Then, 

$$\sup_i\|F_i^m\| \leq \frac{m}{\sqrt{2n}} + o\left(\frac{m}{\sqrt{n}}\right) ,$$

and, similarly,  $$ \inf_i\|F_i^m\| \geq \frac{m}{\sqrt{2n}} - o\left(\frac{m}{\sqrt{n}}\right),$$
both with probability greater than $1-Cne^{-m^{k}}$ for certain positive constants $k$ and $C$.
\end{corollary}

In particular, $\lim\limits_{n\rightarrow\infty} \sup\limits_i\|F_i^m\| \rightarrow 0$ if and only if $m=o(\sqrt{n}).$

\smallskip

As a  consequence of Theorem \ref{main} we prove in Section \ref{Jiangsthingy} a result about the supremum norm $\epsilon_n(m)$.  
\begin{theorem}\label{Jiangnuestro}
For each $n \geq  2$, there exist matrices $Y'_n = (y'_{ij})_{i,j=1}^n$ and 
$U'_n = (u'_{i,j})_{i,j=1}^n$ whose $2n^2$ entries are real random variables defined on the same probability space such that
\begin{itemize}
\item[(i)] the law of $U'_n$ is the normalized Haar measure on the orthogonal group $\mathcal O(n)$; 
\item[(ii)] $\{ y'_{i,j}; 1\leq i, j\leq n\}$ are independent standard normals;
\item[(iii)] set $$\epsilon_n(m)=\max_{1\leq i \leq  n, \, 1 \leq j \leq m}  |\sqrt{n} u'_{i,j}-y'_{i,j}|$$  for  $m = 1, 2, \cdots, n$. Then, there exists $0<\delta<\frac{1}{2}$ such that for any $\varepsilon>0$ we have
\begin{equation*} 
\begin{aligned}
&\epsilon_n(m)\geq(1-\varepsilon)(\sqrt{\varphi(\alpha)} -O(m^{-\delta})) \sqrt{2\log n}  \text{ \hspace{2mm} and}\\
&\epsilon_n(m)\leq(1+\varepsilon)(\sqrt{\varphi(\alpha)} +O(m^{-\delta})) \sqrt{2\log (nm)}\\
\end{aligned}
\end{equation*}
with probability $1-o(1)$, where we consider  $0<\alpha\leq 1$ fixed, $m= [\alpha n]$ and $\varphi(\alpha)=2- \frac{4}{3} \frac{(1-(1 -\alpha)^{3/2})}{\alpha}$ is the function appearing in Theorem \ref{main}.
\end{itemize}
\end{theorem}

If we let $\alpha$ change with $n$ in Theorem \ref{Jiangnuestro} so that  $m_n=o(\frac{n}{\log n})$ we recover the convergence to 0 already obtained in \cite[Theorem 3]{Ji} (see Corollary \ref{Jiangrefinado}). Furthermore, if  we pick $m_n=\frac{\beta n}{\log n}$ we get that $$\sqrt{\beta}\leq
 \epsilon_n(m)\leq \sqrt{2\beta}.$$ Note that in   \cite[Theorem 3]{Ji} the author obtained for this case $ \epsilon_n(m)\rightarrow 2\sqrt{\beta}$. Therefore, we improve the upper bound by a factor $\sqrt{2}$. The key point is that our Theorem \ref{main} allows us to modify the coupling. The price we pay is that we do not obtain an explicit coupling, but a randomized one (with high probability). Also, we have not been able to show that this bound is tight.  Details are given in Section \ref{Jiangsthingy}.

The fact that Theorem \ref{Jiangnuestro} follows from Theorem \ref{main} provides a better understanding of the order $\frac{n}{\log n}$ needed for the convergence of the supremum norm of the block. Roughly, each of the row vectors of the difference, when multiplied from the right by a random unitary, distributes uniformly on the unit sphere. Therefore, the distance between its supremum and Euclidean norms is of the order $\log m$.

One of our original motivations to study this problem was to solve a question in Quantum Information Theory. In that context,  one is interested in the matrices $\Gamma=G_1G_2^{T}$, where $G_1$ and $G_2$ are two independent $n\times m$ real Gaussian matrices, whose rows are normalized in the Euclidean norm. This definition of $\Gamma$ guarantees that it belongs to the set of the so called \emph{quantum correlation matrices}. An interesting question in the context of Bell inequalities is to study whether these matrices belong to the set of \emph{local correlation matrices}, which can be mathematically expressed by the fact that $\Gamma$ belongs to the unit ball of $\ell_\infty^n\otimes_\pi\ell_\infty^n$, where $\pi$ denotes the projective tensor norm. The main result in \cite{GoJiPaVi} states that for a certain constant $\alpha_0\in (0,1)$ one has that $\Gamma$ is nonlocal with probability tending to one as $n$ tends to infinity, provided that $\frac{m}{n}\leq \alpha_0$. The key point in the proof is that a similar result can be obtained from the recent work \cite{Ambainis} if one replaces the Gaussian matrices $G_1$ and $G_2$ above with two truncated matrices independently distributed in the orthogonal group of order $n$. Hence, Theorem \ref{main} is the right ``tool'' to pass from orthogonal to Gaussian matrices, at the cost of a slight modification in the constant $\alpha_0$. We refer to \cite{GoJiPaVi} for details.

The rest of the paper is organized as follows. In Section \ref{Preliminaries} we define our notation and we recall several known facts about the Gaussian distribution that will be repeatedly used later on. Then, in Section \ref{principal} we prove our main result Theorem \ref{main}. At the end of that section we show how simple modifications of the  same techniques prove Corollary \ref{correferee}.  In Section \ref{Jiangsthingy} we apply Theorem \ref{main}  to the study of the supremum norm of the $n\times m$ blocks of  $Y_n-\sqrt{n}U_n$ and we prove Theorem \ref{Jiangnuestro}. Next, we obtain as a corollary a slight improvement of \cite[Theorem 3]{Ji}. 
\section{Preliminaries}\label{Preliminaries}
In this section we define our notation and for the sake of completeness we recall several known facts about the Gaussian measure on $\mathbb R^n$ that will be used several times in the rest of the paper.
We say that a real function $f(n)$ is $O(g(n))$ if there exist constants $C>0$ and $n_0$ such that for all $n>n_0$ we have that  $|f(n)|\leq C g(n)$. We say that $f(n)$ is $o(g(n))$ if $\lim\limits_{n\ra \infty} \frac{f(n)}{g(n)}=0$. We will use $O(g(x))$ and $o(g(x))$ to denote functions on these sets. We will say that a sequence of events $E_n$ holds with probability exponentially small (respectively exponentially close to 1) if there exists $\alpha>0$, independent of $n$ such that  $Pr(E_n)\leq O(e^{-n^{\alpha}})$ (respectively $Pr(E_n)\geq 1- O(e^{-n^{\alpha}})$).

We recall the following well known bounds of the tail of a normal random variable.
\begin{lemma}\label{cotaGaussiana}
Let $Z$ be a standard normal random variable. Then, for every $t>0$,
$$\frac{t}{(1+t^2)\sqrt{2\pi}} e^{-\frac{t^2}{2}} \leq Pr(Z>t)\leq \frac{1}{t\sqrt{2\pi}} e^{-\frac{t^2}{2}}.$$
Hence, for $t\geq 1$, 
$$Pr(Z^2>t^2)\leq  e^{-\frac{t^2}{2}}. $$

We will later choose $t=m^{\frac{\epsilon}{2}}$ to get 
$$Pr(Z^2>m^{\epsilon})\leq  e^{-\frac{m^{\epsilon}}{2}}.$$
\end{lemma}

We will denote the standard Gaussian probability measure (Gaussian measure in short) in $\mathbb R^n$ by $\mathcal G_n$. We will refer to  a  Gaussian vector (matrix)  as a random vector  whose coordinates are independent standard normal random variables in $\mathbb R$.

The following bound of the norm of a Gaussian vector is well known. It can be easily deduced, for instance, from \cite[Lemma 1]{LaurentMassart}.  

\begin{prop}\label{concentracion} 
For every $0<\epsilon<1$, $$\mathcal G_n\{{\bf x}\in \R^n\ :  \ \|{\bf  x}\|_2\geq \frac{\sqrt{n}}{\sqrt{1-\epsilon}}\}\leq e^{\frac{-\epsilon^2 n}{4}} $$
and 
$$\mathcal G_n\{{\bf x}\in \R^n\ : \ \|{\bf x}\|_2\leq \sqrt{n}\sqrt{1-\epsilon}\}\leq e^{\frac{-\epsilon^2 n}{4}}.$$
\end{prop}

We will use several times along the paper the  well known fact that both the Gaussian measure  $\mathcal G_{n}$ in the space of vectors $\mathbb R^n$ and the Gaussian measure $\mathcal G_{n^2}$ in the space of square matrices of order $n$ are biunitarily invariant under the action of the orthogonal group $\mathcal O(n)$. Using this, it is very easy to see that the projection $P_L(\bf{x})$ of a random Gaussian vector $\bf{x}$ onto a fixed  subspace $L$ of dimension $k$ is a Gaussian vector of this subspace.

One can see the rotationally invariant (uniform) measure $\mu_n$ in $S^{n-1}$ as the pushforward measure of $\mathcal G_{n}$ given by the map $f(\bf x)=\frac  {\bf x} {\bf \|x\|}$. That is, given a set $A\subset S^{n-1}$ we have that  $\mu_n(A)=\mathcal G_{n}(f^{-1}(A))$.

Similarly, one can consider the pushforward measure of $\mathcal G_{n^2}$ induced  by the map that takes the first $k$ $n$-dimensional vectors ${\bf x}_1,...,{\bf x}_k \in \mathbb R^{n}$ to the $span\{{\bf x}_1,...,{\bf x}_k\}$, the linear subspace generated by them. This measure is the only one invariant under the action of $\mathcal O(n)$ and therefore we call it the Haar measure in the Grassmannian of the $k$-dimensional subspaces of $\mathbb R^{n}$.

The following proposition follows immediately from the previous explanation.
\begin{prop}\label{concentracion2} 
Let  $L\subset \mathbb R^n$ be a Haar distributed $k$-dimensional subspace and let ${\bf x} \in \mathbb R^n$ be a Gaussian vector independent from $L$. Then, for any $0<\epsilon<1$, 
\begin{equation*}
\label{conslow}Pr \left(\|P_L({\bf x})\|_2\geq \frac{\sqrt{k}}{\sqrt{1-\epsilon}}\right)\leq e^{\frac{-\epsilon^2 k}{4}}
\end{equation*}
and 
\begin{equation*}
\label{consup}Pr \left(\|P_L({\bf x})\|_2 \leq \sqrt{k}\sqrt{1-\epsilon}\right)\leq e^{\frac{-\epsilon^2 k}{4}}.
\end{equation*}
\end{prop}
If we replace the Gaussian vector by a fixed unitary vector we obtain the following estimates.
\begin{prop}\label{concentraciondenormas}
Let  $L\subset \mathbb R^n$ be a Haar distributed $k$-dimensional subspace and let ${\bf y}\in \mathbb R^n$ be a fixed unitary vector. Then, for any $0 < \rho < 1$ we have
$$Pr \left(\|P_L({\bf y})\|\geq \frac{1}{1-\rho}\sqrt{\frac{k}{n}}\right) \leq e^{-\frac{\rho^2k}{4}},$$
and
$$Pr \left(\|P_L({\bf y})\|\leq (1-\rho)\sqrt{\frac{k}{n}}\right) \leq e^{-\frac{\rho^2k}{4}}.$$
For $t>1$ we also have
$$Pr \left(\|P_L({\bf y})\|\geq t \sqrt{\frac{k}{n}} \right) \leq e^{-\frac{k}{4}(t^2-2)}.$$
\end{prop}
\begin{proof}
One can consider a Haar distributed $k$-dimensional subspace $L$ as a Haar distributed  orthogonal matrix $U$ acting on a fixed $k$-dimensional subspace $M$. Hence, $P_L({\bf y})=P_M(U{\bf y})$. Now, the vector $Uy$ is a random uniform vector on the unit sphere of $\mathbb R^n$ and, according to our explanation above,  it is  $\bar{{\bf x}}=\frac{\bf x}{\|{\bf x}\|}$ for a Gaussian vector ${\bf x}$. Thus, $P_L({\bf y})$ has the same distribution as $P_M(\bar{{\bf x}})$. Then, the result can be easily deduced from the known estimates on $P_M(\bar{{\bf x}})$, for example, from \cite[Lemma 2.2]{Gupta}. Also, note that a version of this proposition with slightly  worse constants of the first two bounds can be easily deduced from Proposition \ref{concentracion} and Proposition \ref{concentracion2}.
\end{proof}
\section{Proof of Theorem \ref{main}}\label{principal}
We briefly recall our notation: $Y_n=(y_{i,j})_{i,j=1}^n$ will be a normal Gaussian random matrix.  We consider the column vectors  ${\bf y}_j=(y_{i,j})_{i=1}^n$. With probability 1, they form a basis of $\mathbb R^n$ and, in that case, we can  apply the Gram-Schmidt orthonormalization procedure to $({\bf y}_1,\ldots, {\bf y}_n)$ and we  obtain an orthonormal basis $(\boldsymbol{\nu}_j)_{j=1}^n$. We call $U_n$ the matrix $(\nu_{i,j})_{i,j=1}^n$. 
For every $1\leq m \leq n$ and for every $1\leq i \leq n$, $F_i^m$ is  the vector formed by the the first $m$-coordinates of the $i$th row of $Y_n-\sqrt{n}U_n$.

We start the proof of Theorem \ref{main} with some observations about the Gram-Schmidt orthonormalization process.  Let us examine the situation in step $j$. The Gaussian vectors ${\bf y}_1, \ldots, {\bf y}_{j-1}$ have been chosen independently. Associated to them we have the orthonormal vectors ${\boldsymbol \nu}_1, \ldots, {\boldsymbol \nu}_{j-1}$. Both sets of vectors span the same $j-1$ dimensional subspace $L_{j-1}$. This subspace is distributed according to the Haar measure in the Grassmanian of the  $j-1$ dimensional subspaces of $\mathbb R^n$.

We consider the column vectors $$\Delta_j=\sum_{k=1}^{j-1} \langle  {\bf y}_j , {\boldsymbol \nu}_k\rangle {\boldsymbol \nu}_k =P_{L_{j-1}}({\bf y}_j),$$ 
where $P_{L_{j-1}}$ is the orthogonal projection onto $L_{j-1}$, and we write 
$${\bf y}_j- \sqrt{n} {\boldsymbol \nu}_j=\Delta_j+ ({\bf y}_j-\Delta_j)-\sqrt{n}{\boldsymbol \nu}_j.$$
Let us call $\Delta'_j= ({\bf y}_j-\Delta_j)-\sqrt{n}{\boldsymbol \nu}_j$ and let us note that $({\bf {\bf y}_j}-\Delta_j)$ has the same direction as ${\boldsymbol \nu}_j$ (by definition of ${\boldsymbol \nu}_j$) so that  $$\Delta'_j=(\|{\bf y}_j-\Delta_j\|-\sqrt{n}){\boldsymbol \nu}_j=(\|P_{L_{j-1}^\perp}({\bf y}_j)\|-\sqrt n){\boldsymbol \nu}_j,$$
where $P_{L_{j-1}^\perp}$ is the projection onto the subspace orthogonal to $L_{j-1}$. Note that $\Delta_j$ and $\Delta'_j$ are orthogonal to each other.

Associated to the $\Delta_j$'s and $\Delta'_j$'s, for every $1\leq i \leq n$ and for every $1\leq m\leq n $ we have the (truncated) row vectors $$G_i^m= (\Delta_j(i))_{j=1}^m=\left(\sum_{k=1}^{j-1} \langle {\bf y}_j, {\boldsymbol \nu}_k\rangle \langle {\boldsymbol \nu}_k , {\bf e}_i\rangle\right)_{j=1}^m=\left(\left\langle P_{L_{j-1}}({\bf y}_j), {\bf e}_i\right\rangle\right)_{j=1}^m$$
and
\begin{align*}
H_i^m=(\Delta'_j(i))_{j=1}^m&=\left((\|{\bf y}_j-\Delta_j\|-\sqrt{n})\langle{\boldsymbol \nu}_j,{\bf e}_i\rangle\right)_{j=1}^m\\&=\left( (\|P_{L_{j-1}^\perp}({\bf y}_j)\|-\sqrt n)\langle{\boldsymbol \nu}_j,{\bf e}_i\rangle\right)_{j=1}^m.
\end{align*}
Then, $$\|F_i^m\|^2=\langle G_i^m+ H_i^m, G_i^m+ H_i^m\rangle = \|G_i^m\|^2+\|H_i^m\|^2 + 2 \langle G_i^m, H_i^m\rangle.$$
We will upper and lower bound $\|G_i^m\|$ and $\|H_i^m\|$ outside of  a set of exponentially small probability. Moreover, we show that the leading terms of $\sup_i \|G_i^m\|$ and $\inf_{i} \|G_i^m\|$ are equal (and the same happens for  $\|H_i^m\|$) and thus the bounds are sharp. After that, we will see that $\langle G_i^m, H_i^m\rangle$ is negligible when compared with those bounds outside of a set of probability exponentially small. Finally we will get that $\|F_i^m\|$ is upper and lower bounded by the bounds of $\sqrt{ \|G_i^m\|^2+\|H_i^m\|^2 }$. 

First, we bound $\|G_i^m\|$.
\begin{prop}\label{cotaG}With the notation of Theorem \ref{main}, 
$$\sup_i \|G_i^m\| \leq  \sqrt{\frac{\alpha}{2}m}+O(m^\delta)$$
 and 
$$\inf_{i} \|G_i^m\| \geq \sqrt{\frac{\alpha}{2}m}-O(m^\delta)$$
with probability exponentially close to 1.
\end{prop}

\begin{proof}
We note that $$G_{i, j}^m=\sum_{k=1}^{j-1} \langle {\bf e}_i, {\boldsymbol \nu}_k\rangle \langle {\boldsymbol \nu}_k , {\bf y}_j\rangle=\left \langle \sum_{k=1}^{j-1} \langle {\bf e}_i, {\boldsymbol \nu}_k\rangle {\boldsymbol \nu}_k , {\bf y}_j\right\rangle.$$

Therefore, to obtain the $j$-th coordinate of $G_{i}^m$ we consider the Haar distributed $j-1$ dimensional subspace $L_{j-1}=span\{\boldsymbol{\nu}_1, \ldots, \boldsymbol{\nu}_{j-1}\}=span\{{\bf y}_1,\ldots, {\bf y}_{j-1}\}$. We project ${\bf e}_i$ onto it and we obtain the vector $\sum_{k=1}^{j-1} \langle {\bf e}_i, {\boldsymbol \nu}_k\rangle {\boldsymbol \nu}_k $. {\em Independently}, we consider a random Gaussian vector ${\bf y}_j$ and calculate the inner product $$\left \langle  \sum_{k=1}^{j-1} \langle {\bf e}_i, {\boldsymbol \nu}_k\rangle {\boldsymbol \nu}_k, {\bf y}_j\right\rangle =\left\langle P_{L_{j-1}}({\bf e}_i), {\bf y}_j\right\rangle.$$

The independence of ${\bf y}_j$ with respect to ${\bf y}_1, \ldots,   {\bf y}_{j-1}$ guarantees that 
\begin{equation}\label{distributionequality}
G_{i, j}^m=\left\langle P_{L_{j-1}}({\bf e}_i), {\bf y}_j\right\rangle  \mbox{ is distributed like } \| P_{L_{j-1}}({\bf e}_i)\| g_j= \left( \sum_{k=1}^{j-1} \langle {\boldsymbol \nu}_k,{\bf e}_i\rangle^2\right)^\frac{1}{2} g_j,
\end{equation}
where $g_j$ is a standard normal variable, independent of ${\boldsymbol \nu}_1, \ldots, {\boldsymbol \nu}_{j-1}$,
and, therefore, independent also of all of the previous $g_{j'}$, $j'<j$.

Hence, with the notation $\langle {\boldsymbol \nu}_k,{\bf e}_i\rangle=\nu_{k,i}$ we have that  $\|G_i^m\|^2$ has the same distribution as $\sum_{j=2}^m\sum_{k=1}^{j-1} \nu_{k,i}^2 g_j^2$. 

The fact that  the factors $\sum_{k=1}^{j-1} \nu_{k,i}^2$ multiplying each of the $g_j$'s are not constant and depend on the previous $g_{j'}$'s makes it impossible to apply a concentration bound directly. We circumvent this difficulty by grouping the sum in blocks with a constant factor. This increases the total sum by a very small amount.

We partition the set $\{2,\ldots, m\}$ in $N$ blocks of size $h=\frac{m-1}{N}$. Then, we have 
\begin{equation}\label{particion} \sum_{j=2}^m\sum_{k=1}^{j-1} \nu_{k,i}^2 g_j^2 = \sum_{l=1}^N \sum_{j=(l-1)h+2}^{lh+1} \sum_{k=1}^{j-1} \nu_{k,i}^2 g_j^2\leq 
 \sum_{l=1}^N \left(\sum_{j=(l-1)h+2}^{lh+1} g_j^2 \right)\left( \sum_{k=1}^{lh} \nu_{k,i}^2 \right).
 \end{equation} 
Note that $(\nu_{k,i})_{k=1}^{lh}$ can be seen as the projection of the vector $\bf e_i$ onto a random Haar distributed subspace of dimension $lh$. Now we can apply Proposition \ref{concentracion}, Proposition \ref{concentraciondenormas} and the union bound, and we get that, for every $0<\rho<1$, 
$$\mathcal G_{n^2} \left\{ \|G_i^m\|^2 \geq  \sum_{l=1}^N \left(\frac{1}{(1-\rho)} h\right)\left( \frac{1}{(1-\rho)^2}\frac{lh}{n}\right)\right\} \leq 2N e^{-\frac{\rho^2 h }{4}}.$$ 
Then, using the union bound on the $i$'s and the definitions of $\alpha$ and $N$ we have that, with probability greater than $1-2n\frac m h e^{-\frac{\rho^2 h }{4}}$, 
\begin{equation}\label{concentracionGsup}
\sup_i\|G_i^m\|^2\leq \frac{1}{(1-\rho)^3} \frac{h^2}{n}\frac{N(N+1)}{2} \leq \frac{1}{(1-\rho)^3} \frac{\alpha}{2}\left(m+h\right).
\end{equation}
Different choices of $h, \rho$ yield now different versions of our result. For instance, we can choose $h=m^{1/2}$, $\rho=m^{-1/5}$ and we have $\|G_i^m\|^2\leq \frac{\alpha}{2}m+ O(m^{4/5})$ with probability $1-2n\sqrt m e^{-\frac{m^{1/10}}{4}}$. Easy computations show that $\|G_i^m\|^2\leq \frac{\alpha}{2}m+ O(m^{2/3+\epsilon})$ with probability tending to $0$ exponentially in $m^\epsilon$.

We can also choose $h=\frac{\epsilon}{2} m$ and $\rho=\frac{\epsilon}{8}$ and, using the Taylor expansion of $\frac{1}{(1-\rho)^3}$, we get that
$$\|G_i^m\|^2\leq (1+ \epsilon) \frac{\alpha}{2}m,$$
with probability greater than $1-\frac {4 n} \epsilon e^{-\frac{\epsilon^3 m }{2^9}}$. 

For the sake of clarity, we have written   Equation (\ref{particion}) as if $N$ and $h$ were always integers. If they were not, we  can take $N'=[N]+1$ and $h'=[h]+1$. This adds at most $[N]+[h]+1$ terms to the previous sum. It is very easy to see that this extra terms do not change the above estimates. 

This upper bounds $\|G_i^m\|$. 

Similar reasonings prove the lower bound. To do this, one replaces Equation (\ref{particion}) by $$\sum_{j=2}^m\sum_{k=1}^{j-1} \nu_{k,i}^2 g_j^2\geq \sum_{l=2}^N \sum_{j=(l-1)h+2}^{lh+1} \sum_{k=1}^{(l-1)h+1} \nu_{k,i}^2 g_j^2,$$ 
and proceeds similarly as with the upper bound. \end{proof}

We proceed now to bound $\|H_i^m\|$. 
\begin{prop}\label{cotasH}
With the notation of Theorem \ref{main}, 
$$\sup_i \|H_i^m\| \leq  \sqrt{\left(2- \alpha/2-\frac{4}{3} \frac{(1-(1 -\alpha)^{3/2})}{\alpha}\right)m} +O(m^{\delta})$$
and 
$$\inf_i \|H_i^m\| \geq  \sqrt{\left(2- \alpha/2-\frac{4}{3} \frac{(1-(1 -\alpha)^{3/2})}{\alpha}\right)m} -O(m^{\delta}),$$
with probability exponentially close to 1.
\end{prop}
\begin{proof}
Recall that $$H_i^m=\left((\|{\bf y}_j-\Delta_j\|-\sqrt{n})\nu_{j,i}\right)_{j=1}^m,$$
where ${\bf y}_j-\Delta_j$ is the projection of ${\bf y}_j$ onto the $n-(j-1)$ dimensional subspace orthogonal to the subspace $L_{j-1}=span\{{\bf y}_1, \ldots, {\bf y}_{j-1}\}$. We will first bound  the term $(\|{\bf y}_j-\Delta_j\|-\sqrt{n})^2=\left(\|P_{L^\perp_{j-1}}{\bf y}_{j}\|-\sqrt n\right)^2$. For that, we need an auxiliar Lemma which we will also use later.
\begin{lemma}\label{yaveremos}
With the notation above, we have 
\begin{itemize}
\item[(i)] For every $0<\rho<1$ and for every $m<n$
\begin{equation*} \left(\sqrt n-\|P_{L^\perp_{j-1}}{\bf y}_{j}\|\right)^2 \leq \left(\sqrt{n} -(1-\rho)^{\frac{1}{2}}\sqrt{n-j+1}\right)^2 \text{     }\text{  for    } \text{    }1\leq j\leq m, 
\end{equation*}
except for a set $Z_1$ with $\mathcal G_{n^2}(Z_1)\leq  m \left( e^{-\frac{\rho^2 (n-m+1)}{4}}+ e^{-\frac{\rho^2 (n-m+1)}{16}} \right)$.
\item[(ii)]
Let $0<\rho_0<1$ and $j_0\in \mathbb N$ such that $(1-\rho_0)^{-1}(n-j_0+1)\leq n$. Then, for every $m<n$ 
\begin{equation*} \left(\sqrt n-\|P_{L^\perp_{j-1}}{\bf y}_{j}\|\right)^2\geq \left(\sqrt{n} -(1-\rho_0)^{-\frac{1}{2}}\sqrt{n-j+1}\right)^2 \text{     }\text{  for    } \text{    } j_0< j\leq m,
\end{equation*}
except for a set $Z_2$ with $\mathcal G_{n^2}(Z_2)\leq  (m-j_0) e^{-\frac{\rho_0^2 (n-m+1)}{4}}$.\end{itemize}
\end{lemma}
\begin{proof}
First we prove (i). We choose $\epsilon=\rho/2$ in Equation (\ref{conslow}) and $\epsilon=\rho$ in Equation (\ref{consup}) and  we get  
\begin{equation*} (1-\rho)^{\frac{1}{2}}\sqrt{n-j+1}-\sqrt{n}\leq \|P_{L^\perp_{j-1}}{\bf y}_{j}\|-\sqrt n \leq \left(1-\frac \rho 2\right)^{-\frac{1}{2}}\sqrt{n-j+1}-\sqrt{n} 
\end{equation*}
except for a set of measure $e^{\frac{-\rho^2 (n-j+1)}{4}}+e^{\frac{-\rho^2 (n-j+1)}{16}}$. Using the fact that for $0<\rho<1$ $1-(1-\rho)^{\frac{1}{2}}\geq (1-\frac{\rho}{2})^{\frac{-1}{2}}-1$ we have 
$$\left|(1-\frac{\rho}{2})^{\frac{-1}{2}}\sqrt{n-j+1}-\sqrt{n}\right| \leq \left|(1-\rho)^{\frac{1}{2}}\sqrt{n-j+1}-\sqrt{n}\right|.$$ Then, taking squares and applying a union bound we get (i). 

The proof of (ii) follows from Equation (\ref{conslow}), the extra condition on $\rho_0$ and $j_0$ and the union bound.
\end{proof}

Now, in order to upper bound $\|H_i^m\|^2$ we first consider the case $\alpha<1$. As in the proof of Theorem \ref{cotaG}, we partition the set $\{1,\ldots, m\}$ in $N$ blocks of size $h=\frac{m}{N}$. Then, using Lemma \ref{yaveremos}.(i),  we write
\begin{align} 
\label{particionH} \|H_i^m\|^2&=\sum_{j=1}^m \left(\|{\bf y}_j-\Delta_j\|-\sqrt{n}\right)^2 \nu_{j,i}^2  \\
&\leq \sum_{l=1}^N \sum_{j=(l-1)h+1}^{lh} \left(\sqrt{n} -(1-\rho)^{\frac{1}{2}}\sqrt{n-lh+1}\right)^2 \nu_{j,i}^2 \nonumber \\
&=\sum_{l=1}^N \left(\sqrt{n} -(1-\rho)^{\frac{1}{2}}\sqrt{n-lh}\right)^2 \sum_{j=(l-1)h+1}^{lh}  \nu_{j,i}^2, \nonumber
\end{align} 
outside of $Z_1$. 

On the other hand, considering $(\nu_{k,i})_{i=(l-1)h+1}^{lh}$ as the projection of ${\bf e}_i$ onto a random subspace of dimension $lh$, Proposition \ref{concentraciondenormas}  tells us that, for every $1\leq i \leq n$ and $1\leq l\leq N$, $$\sum_{j=(l-1)h+1}^{lh}\nu_{j,i}^2 \leq   \frac{1}{(1-\rho')^2}\frac{h}{n}$$
except for a set $Z_1'$ with $\mathcal G_{n^2}(Z_1')\leq nN e^{-\frac{\rho'^2 h}{4}}.$

So, we have that, outside of $Z_1\cup Z_1' $, 
\begin{align*}
\|H_i^m\|^2 \leq \frac{1}{(1-\rho')^2}\frac{h}{n} &\sum_{l=1}^N \left(\sqrt{n} -(1-\rho)^{\frac{1}{2}}\sqrt{n-lh}\right)^2\\
=\frac{1}{(1-\rho')^2} \frac{h}{n} &\left[nN+ (1-\rho)\left(nN-h\frac{N(N+1)}{2}\right) \right. \\
 &-\left. 2(1-\rho)^{\frac{1}{2}}\sqrt{n} \sum_{l=1}^N \sqrt{n -lh}\right].
\end{align*} 
We can bound
$$\sum_{l=1}^N\sqrt{n-lh}\geq \int_1^{N}\sqrt{n-xh} dx=\frac {2} {3h}\left((n- h)^{3/2}-  (n-N h)^{3/2} \right).$$
Then, putting this together with the definitions of $\alpha$ and $N$, we get that
\begin{equation*}
\begin{aligned}
\|H_i^m\|^2 \leq \frac{1}{(1-\rho')^2} m & \left[ 1+ (1-\rho)\left(1-\frac \alpha 2 -\frac{\alpha h}{2m}\right) \right.  \\
& \left. -  (1-\rho)^{\frac{1}{2}} \frac {4} {3\alpha }\left(\left(1- \frac {\alpha h} m\right)^{3/2}-  (1-\alpha)^{3/2}\right)\right],
\end{aligned}
\end{equation*}
with probability greater than
\begin{equation}\label{concentracionmenor1} 1-\frac{m^2}{\alpha h} e^{-\frac{\rho^2 h}{4}}-m \left( e^{-\frac{\rho^2 (n-m)}{4}}+ e^{-\frac{\rho^2 (n-m)}{16}} \right).
\end{equation}

Again, different choices of $h, \rho, \rho'$ yield now different versions of our result. For instance, taking $h=m^{1/2}$, $\rho=\rho'=m^{-1/5}$ we get 
$$\|H_i^m\|^2 \leq \left(2- \alpha/2-\frac{4}{3} \frac{(1-(1 -\alpha)^{3/2})}{\alpha}\right)m +O(m^{4/5}),$$
with probability tending to one exponentially in $m^{1/10}$.
Easy computations also show that $\|G_i^m\|^2\leq \frac{\alpha}{2}m+ O(m^{2/3+\epsilon})$ with probability tending to 0 exponentially in $m^\epsilon$.

The reasonings above do not apply directly to the case $\alpha=1$,  as the bound of the probability in Equation (\ref{concentracionmenor1}) becomes trivial in that case. To overcome this issue, in case $\alpha=1$ we consider $h=\frac {n-\sqrt n} N$ and rewrite Equation (\ref{particionH}) as 
\begin{align*} \|H_i^m\|^2& \leq \sum_{j=1}^n \big(\|{\bf y}_j-\Delta_j\|-\sqrt{n} \big)^2 \nu_{j,i}^2 \\ 
& =  \sum_{j=1}^{n-\sqrt{n}}  \big(\|{\bf y}_j-\Delta_j\|-\sqrt{n} \big)^2 \nu_{j,i}^2 +  \sum_{j=n-\sqrt{n}+1}^n  \big(\|{\bf y}_j-\Delta_j\|-\sqrt{n} \big)^2 \nu_{j,i}^2\\ &\leq \sum_{l=1}^N \sum_{j=(l-1)h+1}^{lh} \left(\sqrt{n} -(1-\rho)^{\frac{-1}{2}}\sqrt{n-lh+1}\right)^2 \nu_{j,i}^2 + \sum_{j=n-\sqrt{n}+1}^n n \nu_{j,i}^2,
\end{align*}
outside of the set $Z_1$ defined in Lemma \ref{yaveremos} in the case $m=n-\sqrt{n}$. 

The first summand is treated as previously where now $m=n-\sqrt{n}$. We note that, using Proposition \ref{concentraciondenormas} and the union bound once again, the second summand verifies, with probability greater than $1-ne^{-\frac{\rho^2 \sqrt n}{4}}$,
$$n \sum_{j=n-\sqrt{n}+1}^n  \nu_{j,i}^2\leq n \frac{1}{(1-\rho^2)\sqrt{n}}=\frac{\sqrt{n}}{(1-\rho^2)}.$$
This only adds an $O(\sqrt{n})$ term which does not modify the result. 
This finishes the proof of the upper bound. 
\smallskip

For the lower bound we reason similarly. 
Consider $j_0,\rho_0$ as in Lemma \ref{yaveremos}.(ii). Then, with probability $1-n(m-j_0) e^{-\frac{\rho_0^2 (n-m+1)}{4}}$, for every $1\leq i\leq n$ we have that 
\begin{equation}\label{sumaj0}
\begin{aligned}
\|H_i^m\|^2  \geq &\sum_{j=j_0+1}^m \left(\|{\bf y}_j-\Delta_j\|-\sqrt{n}\right)^2 \nu_{j,i}^2\\ 
\geq &\sum_{j=j_0+1}^m  \left(\sqrt{n} -(1-\rho_0)^{-1/2}\sqrt{n-j+1}\right)^2 \nu_{j,i}^2.
\end{aligned}
\end{equation}

Partitioning the set $\{j_0+1,\ldots, m\}$ in $N$ blocks of size $h=\frac{m-j_0}{N}$ and using similar reasonings  to those used for the upper bound in the case $\alpha<1$ we obtain

\begin{align*}
\inf_i \|H_i^m\|^2  \geq (1-\rho')^2& (m-j_0) \left[1+ (1-\rho_0)  \left(1-\frac \alpha 2+\frac{\alpha(j_0+h)}{2m}\right) \right.\\ 
&\left.-\frac {4(1-\rho_0)^{\frac{1}{2}}} {3(\alpha -\frac{\alpha j_0}{m})}\left(\left(1+\frac {\alpha h} m\right)^{3/2}-  \left(1-\alpha+\frac {\alpha j_0} m\right)^{3/2}\right)\right],
\end{align*}
with probability higher than $1-\frac m \alpha (m-j_0) e^{-\frac{\rho_0^2 (n-m+1)}{4}}-\frac h \alpha e^{-\frac{\rho'^2 h}{4}}$. As in expressions (\ref{concentracionGsup}) and (\ref{concentracionmenor1}) different values of $j_0$, $\rho_0$, $\rho$ and $h$ give different bounds.

The case $\alpha=1$ can be treated as in the upper bound. The terms in (\ref{sumaj0}), where $n-\sqrt n+1\leq j \leq n$ can only add up to  $O(\sqrt n)$ and the rest of the terms  can be bounded as before. 

Note that, as in the proof of Proposition \ref{cotaG}, we are assuming that $N$, $h$ and $\sqrt n$ are integers. If this is not the case we can consider $N'=[N]+1$ and $h'=[h]+1$ for the upper estimates ($N'=[N]$ and $h'=[h]$ for the lower estimates) and $[\sqrt n]$ for the case $\alpha=1$. Adding or subtracting these extra terms will give negligible quantities compared with the sums.
This finishes the proof.
\end{proof}

We now need to prove that $\langle G_i^m,H_i^m\rangle$ is negligible when compared with $\|G_i^m\|^2$ and $\|H_i^m\|^2$. More precisely, we will use similar techniques to show that $\langle G_i^m,H_i^m\rangle$ is, with probability exponentially close to 1, of smaller order in $m$ than $\|G_i^m\|^2$ and $\|H_i^m\|^2$. As shown in Proposition \ref{cotaG} and Proposition \ref{cotasH} above, each of them is $\Theta(m)$. That is, there exist constants $k_1,k_2$ and $m_0$ such that for all $m>m_0$, we have that $k_1 m\leq\|G_i^m\|^2 \leq k_2 m$ and analogously for $\|H_i^m\|^2$.
\begin{prop}\label{sharp}With the notation of Theorem \ref{main}, given $\epsilon >0$ we have 
$$\langle G_i^m,H_i^m\rangle=O(m^{\frac{1}{2}+\epsilon}),$$
with probability exponentially close to 1. 
\end{prop}
This proposition, together with Propositions \ref{cotaG} and \ref{cotasH}, finishes the proof of Theorem \ref{main}.

For the sake of clarity we will first show the following technical lemma that will be used in the proof of Proposition \ref{sharp}. 
\begin{lemma}\label{aux1}Let ${\bf y}_j$ be a Gaussian vector and $L_{j-1}$ a Haar distributed subspace of dimension $j-1$, then   
$$Pr\left(\langle P_{L_{j-1}}({\bf y}_j), {\bf e}_i\rangle^2> \frac{j-1}{n} m^\epsilon\right)\leq   2e^{-\frac{m^{\epsilon/2} -2}{4}}.$$
\end{lemma}

\begin{proof}
First of all note that $\langle P_{L_{j-1}}({\bf y}_j), {\bf e}_i\rangle=\langle P_{L_{j-1}}({\bf e}_i),{\bf y}_j\rangle$. 
We have already shown in Equation (\ref{distributionequality}) that $\langle P_{L_{j-1}}({\bf e}_i),{\bf y}_j\rangle^2$ has the same distribution as the term  $\|P_{L_{j-1}}({\bf e}_i)\|^2 g_j^2$, where $g_j$ is a standard normal  variable. 
Putting together Lemma \ref{cotaGaussiana}, Proposition \ref{concentraciondenormas} and the union bound, we get 
$$Pr\left(\langle P_{L_{j-1}}({\bf e}_i), {\bf y}_j\rangle^2> \frac{j-1}{n} m^\epsilon \right) \leq e^{-\frac{j-1}{4}(m^{\epsilon/2}-2)}+e^{-\frac{m^{\epsilon/2}}{2}} \leq 2e^{-\frac  {m^{\epsilon/2}-2} {4}}.$$\end{proof}
We will also need Hoeffding's inequality \cite{Hoeffding}. 
\begin{prop}[Hoeffding's inequality]\label{Hoeffding}
Let $(X_i)_{i=1}^n$ be a family of independent random variables such that $a_i\leq X_i\leq b_i$ for $i=1,...,n$. Let $S=\sum_{i=1}^n X_i$. Then, for every $a>0$,  $$Pr(|S-\mathbb{E}(S)|>a)\leq 2e^{-\frac{2a^2}{\sum_i (b_i-a_i)^2}}.$$ 
\end{prop}
\begin{proof}[Proof of Proposition \ref{sharp}]
Recall that we have $$G_i^m= \left(\sum_{k=1}^{j-1} \langle {\bf y}_j, {\boldsymbol \nu}_k\rangle \langle {\boldsymbol \nu}_k , {\bf e}_i\rangle\right)_{j=1}^m=\left(\langle P_{L_{j-1}}({\bf y}_j), {\bf e}_i\rangle\right)_{j=1}^m$$
and 
$$H_i^m=\left((\|{\bf y}_j-\Delta_j\|-\sqrt{n})\langle{\boldsymbol \nu}_j,{\bf e}_i\rangle\right)_{j=1}^m=
\left((\| P_{L^{\perp}_{j-1}}({\bf y}_j) \|-\sqrt{n})\langle{\boldsymbol \nu}_j,{\bf e}_i\rangle\right)_{j=1}^m.$$
Therefore,
\begin{align}\label{Initial orthogonal}
\langle G_i^m,H_i^m\rangle&=\sum_{j=1}^m \langle P_{L_{j-1}}({\bf y}_j), {\bf e}_i\rangle
\big(\| P_{L^{\perp}_{j-1}}({\bf y}_j\big ) \|-\sqrt{n}\big)\langle{\boldsymbol \nu}_j,{\bf e}_i\rangle\\&\nonumber=\sum_{j=1}^m \left|\langle P_{L_{j-1}}({\bf y}_j), {\bf e}_i\rangle\right|
\big(\| P_{L^{\perp}_{j-1}}({\bf y}_j\big ) \|-\sqrt{n}\big)\langle{\boldsymbol \nu}_j,{\bf e}_i\rangle \sign\left(\langle P_{L_{j-1}}({\bf y}_j), {\bf e}_i\rangle\right).
\end{align}
We claim that the probability distribution of the previous expression is the same as the probability distribution of
\begin{align}\label{split variables}
\sum_{j=1}^m \big|\langle P_{\tilde{L}_{j-1}}(\tilde{{\bf y}}_j), {\bf e}_i\rangle\big|
\big(\| P_{\tilde{L}^{\perp}_{j-1}}(\tilde{{\bf y}}_j\big ) \|-\sqrt{n}\big)\langle\tilde{{\boldsymbol \nu}}_j,{\bf e}_i\rangle\epsilon_j,
\end{align}where $\tilde{{\bf y}}_1,\cdots, \tilde{{\bf y}}_n$ are independent Gaussian vectors, $\tilde{{\boldsymbol \nu}}_1, \cdots , \tilde{{\boldsymbol \nu}}_n$ are the correspon\-ding vectors obtained from the Gram-Schmidt orthonormalization procedure and $\epsilon_1,\cdots ,\epsilon_n$ are independent and identically distributed Bernoulli variables taking values in an independent probability space.

In order to see this, let us consider the space $\mathbb R^n\times \stackrel{(n)}{\cdots}\times \mathbb R^n$ with the Gaussian measure $\mathcal G_{n}$ on each $\mathbb R^n$. For each $j$ we denote by ${\bf z}_j=(z_{k,j})_{k=1}^n$ the Gaussian vector in the corresponding copy of $\mathbb R^n$. For each $j\geq2$ we consider in  $\mathbb R^n$ the equivalence relation ${\bf z}\sim_j {\bf z'}$ if and only if $(z_1,\ldots, z_{j-1})=\pm (z'_1,\ldots, z'_{j-1})$ and $(z_j,\ldots, z_{n})= (z'_j,\ldots, z'_{n})$. Then, $\mathbb R^n=\left(\mathbb R^n/\sim_j\right) \times \{-1,1\}$, with the identification ${\bf z}_j=([{\bf z}_j],\sigma_j)$, with $\sigma_j\in \{-1, 1\}$. We define the probability measure $\mathcal G'_{n}$ on  $\mathbb R^n/\sim_j$ by the density $f'([{\bf z}_j])=2f({\bf z}_j)$, where $f$ is the density of $\mathcal G_{n}$, and we call $\mu$ the uniform probability on $\{-1,1\}$. We clearly have $\mathcal G_{n}=\mathcal G'_{n}\otimes \mu$.

Let us now consider a family of independent Gaussian vectors $({\bf z}_1,\cdots ,{\bf z}_n)$ in the previous probability space. 
For $j=1$ there is no equivalence relation, and we define $\tilde{{\bf y}}_1={\bf z}_1$, which is clearly a Gaussian vector. Consequently, we define $\tilde{{\boldsymbol \nu}}_1=\frac{1}{\|\tilde{{\bf y}}_1\|}\tilde{{\bf y}}_1$. 

Now, for each $2\leq j \leq m $, we consider $\tilde{L}_{j-1}$, the random $(j-1)$-dimensional subspace spanned by $\tilde{{\bf y}}_1, \ldots, \tilde{{\bf y}}_{j-1}$. The vectors $\tilde{{\boldsymbol \nu}}_1 ,\ldots, \tilde{{\boldsymbol \nu}}_{j-1}$ form an orthonormal basis of $\tilde{L}_{j-1}$. Hence, we can complete this set to obtain a basis of $\mathbb R^n$, $ \{\tilde{{\boldsymbol \nu}}_1,\ldots, \tilde{{\boldsymbol \nu}}_{j-1}, {\boldsymbol \nu}^*_j, \ldots, {\boldsymbol \nu}^*_n\}$. The added vectors needed to complete the orthonormal basis can be chosen at will. In general they will change as $j$ changes.  

Let us denote by $U_j$ the orthogonal matrix whose columns are the vectors of the previous basis. Then, given ${\bf z}_j$, we define $\tilde{{\bf y}}_j=U_j{\bf z}_j$. Since the orthogonal matrix $U_j$ is independent of the Gaussian vector ${\bf z}_j$, we immediately deduce that $\tilde{{\bf y}}_j$ is a Gaussian vector independent of $\tilde{{\bf y}}_1,\cdots ,\tilde{{\bf y}}_{j-1}$. 

From the Gram-Schmidt orthonormalization procedure we have that 
$$\tilde{{\boldsymbol \nu}}_j=\frac{\tilde{{\bf y}}_j-{\sum_{k=1}^{j-1}\langle \tilde{{\bf y}}_j, \tilde{{\boldsymbol \nu}}_k\rangle \tilde{{\boldsymbol \nu}}_k}}{\|\tilde{{\bf y}}_j-{\sum_{k=1}^{j-1}\langle \tilde{{\bf y}}_j, \tilde{{\boldsymbol \nu}}_k\rangle \tilde{{\boldsymbol \nu}}_k}\|}.$$ 

It follows immediately that  $$P_{\tilde{L}_{j-1}}(\tilde{{\bf y}}_j)=\sum_{k=1}^{j-1} z_{k,j} \tilde{{\boldsymbol \nu}}_k\text{      }\text{      and    }\text{      } P_{\tilde{L}^{\perp}_{j-1}}(\tilde{{\bf y}}_j)=\sum_{k=j}^{n} z_{k,j} {\boldsymbol \nu}^*_k,$$
where we recall that ${\bf z}_j=(z_{k,j})_{k=1}^n$. 

With the identification ${\bf z}_j=([{\bf z}_j], \sigma_j)$, it is easy to  see that $\big|\langle P_{\tilde{L}_{j-1}}(\tilde{{\bf y}}_j), {\bf e}_i\rangle\big|$ and $\big(\| P_{\tilde{L}^{\perp}_{j-1}}(\tilde{{\bf y}}_j\big ) \|-\sqrt{n}\big)\langle\tilde{{\boldsymbol \nu}}_j,{\bf e}_i\rangle$ do not depend on  $\sigma_2, \ldots ,\sigma_n$ (or, equivalently, they only depend on the variables $[{\bf z_j}]$). Indeed, to see this we note first that it follows from the definitions that  the vectors $\tilde{{\boldsymbol \nu}}_j$ are independent of $\sigma_2, \cdots ,\sigma_n$. Next, we note that the dependence of  $\langle P_{\tilde{L}_{j-1}}(\tilde{{\bf y}}_j), {\bf e}_i\rangle$ with respect to $\sigma_2, \ldots ,\sigma_n$ is cancelled out by the absolute value. 

Hence,  expression (\ref{Initial orthogonal}) applied to the independent Gaussian vectors $\tilde{{\bf y}}_1,\cdots, \tilde{{\bf y}}_n$ has the form
\begin{align*}
\sum_{j=1}^m \big|\langle P_{\tilde{L}_{j-1}}(\tilde{{\bf y}}_j), {\bf e}_i\rangle\big|
\big(\| P_{\tilde{L}^{\perp}_{j-1}}(\tilde{{\bf y}}_j\big ) \|-\sqrt{n}\big)\langle\tilde{{\boldsymbol \nu}}_j,{\bf e}_i\rangle\epsilon_j,
\end{align*} where $\tilde{{\bf y}}_1,\cdots, \tilde{{\bf y}}_n$ are Gaussian vectors independent of $\sigma_2, \ldots ,\sigma_n$ and $\epsilon_j=\epsilon_j(\sigma_j)$ are independent identically distributed Bernoulli variables.

Equation (\ref{split variables}) will allow us to apply Proposition \ref{Hoeffding}: For fixed $({\bf z}_1, [{\bf z_2}], \ldots, [{\bf z}_n])$ we can consider the independent random variables (function of  $(\sigma_2,\cdots ,\sigma_n)$)
\begin{align*}
\Big(\big|\langle P_{\tilde{L}_{j-1}}(\tilde{{\bf y}}_j), {\bf e}_i\rangle\big|
\big(\| P_{\tilde{L}^{\perp}_{j-1}}(\tilde{{\bf y}}_j\big ) \|-\sqrt{n}\big)\langle\tilde{{\boldsymbol \nu}}_j,{\bf e}_i\rangle\epsilon_j(\sigma_j)\Big)_{j=1}^m,
\end{align*}
Then, Proposition \ref{Hoeffding} gives us that, for fixed $({\bf z}_1, [{\bf z_2}], \ldots, [{\bf z}_n])$.
\begin{equation}\label{cotachernoff}
\mu^{\otimes m} \Big(|\langle G_i^m,H_i^m\rangle| \geq a \Big)\leq 2e^{-\frac{a^2}{2\sum_{j=1}^m \langle P_{\tilde{L}_{j-1}}(\tilde{{\bf y}}_j), {\bf e}_i\rangle^2  (
\| P_{\tilde{L}^{\perp}_{j-1}}(\tilde{{\bf y}}_j) \|-\sqrt{n})^2 \langle \tilde{{\boldsymbol \nu}}_j,{\bf e}_i\rangle^2}}.
\end{equation}

We consider first the case $\alpha<1$. 

It follows from Lemma \ref{aux1}, Proposition \ref{concentraciondenormas}, Lemma \ref{yaveremos} and a union bound argument that 
\begin{align}\label{eqalpha1}&2\sum_{j=1}^m \langle P_{\tilde{L}_{j-1}}(\tilde{{\bf y}}_j), {\bf e}_i\rangle^2  \big(\| P_{\tilde{L}^{\perp}_{j-1}}(\tilde{{\bf y}}_j) \|-\sqrt{n}\big)^2 \langle \tilde{{\boldsymbol \nu}}_j,{\bf e}_i\rangle ^2 \leq 
\\ &\nonumber \leq 2\sum_{j=1}^m m^\epsilon \frac{j-1}{n}\big(\sqrt{n} -(1-\rho)^{\frac{1}{2}}\sqrt{n-j+1}\big)^2\frac{m^\epsilon}{n}\leq m^{2\epsilon+1}
\end{align} 
with probability larger than $$1-C(m,\rho,\epsilon):=1-2me^{-(\frac{m^{\frac{\epsilon}{2}}-2}{4})}+m \left( e^{-\frac{\rho^2 ((1/\alpha-1)m+1)}{4}}+ e^{-\frac{\rho^2 ((1/\alpha-1)m+1)}{16}} \right)+ \sqrt{e}me^{\frac{1}{4}m^{2\epsilon}}.$$
(The second inequality in the second line of (\ref{eqalpha1}) follows from easy calculations).

We note now that $Pr \Big(|\langle G_i^m,H_i^m\rangle| \geq a \Big)$ is upper bounded by
\begin{align*}
Pr \Big(|\langle G_i^m,H_i^m\rangle| \geq a \,\Big| \,  2\sum_{j=1}^m \langle P_{\tilde{L}_{j-1}}(\tilde{{\bf y}}_j), {\bf e}_i\rangle^2  \big(\| P_{\tilde{L}^{\perp}_{j-1}}(\tilde{{\bf y}}_j) \|-\sqrt{n}\big)^2 \langle \tilde{{\boldsymbol \nu}}_j,{\bf e}_i\rangle ^2 \leq m^{2\epsilon+1} \Big)\\ + Pr\Big( 2\sum_{j=1}^m \langle P_{\tilde{L}_{j-1}}(\tilde{{\bf y}}_j), {\bf e}_i\rangle^2  \big(\| P_{\tilde{L}^{\perp}_{j-1}}(\tilde{{\bf y}}_j) \|-\sqrt{n}\big)^2 \langle \tilde{{\boldsymbol \nu}}_j,{\bf e}_i\rangle ^2 > m^{2\epsilon+1} \Big),
\end{align*}where we denote by $Pr(A|B)$  the probability of the event $A$ conditioned to $B$.

We pick $\epsilon'>\epsilon$ and we fix $a=m^{\frac{1}{2}+\epsilon'}$. Then, Equations (\ref{cotachernoff}) and (\ref{eqalpha1}) imply that  
\begin{equation*}
Pr \Big(|\langle G_i^m,H_i^m\rangle| \geq m^{\frac{1}{2}+\epsilon'} \Big)\leq 2e^{-\frac{m^{1+2\epsilon'}}{m^{2\epsilon}+1}}+ C(m,\rho,\epsilon)=2e^{m^{-2(\epsilon'-\epsilon)}}+C(m,\rho,\epsilon),
\end{equation*}which tends exponentially fast to zero as $n$ grows to infinity.

The case $\alpha=1$ has to be considered separately as the bound in the concentration of Lemma \ref{yaveremos} becomes trivial. In order to overcome this issue we reason as in the proof of Proposition \ref{cotasH}. We can divide the sum in (\ref{eqalpha1}) in two terms
\begin{align*}
&2\sum_{j=1}^{n-\sqrt{n}} \langle P_{L_{j-1}}({\bf y}_j), {\bf e}_i\rangle^2  \left(
\| P_{L^{\perp}_{j-1}}({\bf y}_j) \|-\sqrt{n}\right)^2 \langle{\boldsymbol \nu}_j,{\bf e}_i\rangle^2 \\
+&2\sum_{j=n-\sqrt{n}+1}^n \langle P_{L_{j-1}}({\bf y}_j), {\bf e}_i\rangle^2  \left(
\| P_{L^{\perp}_{j-1}}({\bf y}_j) \|-\sqrt{n}\right)^2 \langle{\boldsymbol \nu}_j,{\bf e}_i\rangle^2,
\end{align*}
where the first summand is treated as previously giving an upper bound of $(n-\sqrt n)^{2\epsilon+1}$ and the second is $O(n^{1/2+2\epsilon})$, which is negligible compared with the first. Proceeding as in the case $\alpha<1$ the result follows.
\end{proof}

\smallskip

We prove now Corollary \ref{correferee}. Before writing the actual proof, we note that considering the  Taylor expansion of the expression in Theorem \ref{main} we obtain  $$2- \frac{4}{3} \frac{(1-(1 -\alpha)^{3/2})}{\alpha}=
\frac{\alpha}{2}+  \frac{\alpha^2}{12}+ O(\alpha^3),$$
so that for $\alpha$ close enough to 0 we would expect $$\|F_i^m\|\approx\sqrt{\left(2- \frac{4}{3} \frac{(1-(1 -\alpha)^{3/2})}{\alpha}\right)m}\approx \sqrt{\frac{\alpha m}{2}}=\sqrt{\frac{m^2}{2n}},$$
as it is indeed the case. 
\begin{proof}[\bf Proof of Corollary \ref{correferee}:] As in the proof of Theorem \ref{main} we consider the decomposition $F_i^{m}=G_i^{m}+ H_i^{m}$ and we treat separately the norms of $G_i^{m}$ and $H_i^{m}$. 

\smallskip

We start with $\|G_i^{m}\|$. Recall that Equation (\ref{concentracionGsup}) says that,  with probability greater than $1-2n\frac m h e^{-\frac{\rho^2 h }{4}}$, 
\begin{equation*}
\sup_i\|G_i^m\|^2\leq \frac{1}{(1-\rho)^3} \frac{m}{2n}\left(m+h\right).
\end{equation*}

If we make, for instance,  the choices $\rho=m^{-\frac{1}{6}} $, $h=m^\frac{1}{2}$, and take into account that  $\frac{1}{(1-\rho)^3}\leq 1+4\rho$ for $\rho$ small enough, we get that 
\[\sup_i\|G_i^m\|^2\leq \frac{m^2}{2n} + O\left(\frac{m^{\frac{11}{6}}}{n}\right),\]
with probability greater than $1-2n m^\frac{1}{2} e^{-\frac{m^\frac{1}{6}}{4}}$.

Similar reasonings show that 
\[\inf_i\|G_i^m\|^2\geq \frac{m^2}{2n} - O\left(\frac{m^{\frac{11}{6}}}{n}\right),\]
with probability greater than $1-2n m^\frac{1}{2} e^{-\frac{m^\frac{1}{6}}{4}}$.

\smallskip

Now we bound the norm of $H_i^m$. It follows from the proof of Proposition \ref{cotasH} (see Equation (\ref{particionH}) and Lemma \ref{yaveremos}) that
\begin{align*} 
\|H_i^m\|^2&=\sum_{j=1}^m \left(\|{\bf y}_j-\Delta_j\|-\sqrt{n}\right)^2 \nu_{j,i}^2  \\
&\leq \sum_{j=1}^m  \left(\sqrt{n} -(1-\rho)^{\frac{1}{2}}\sqrt{n-j+1}\right)^2 \nu_{j,i}^2 
\end{align*} holds for every $i=1,\dots,n$, 
outside of a set 
$Z_1$ with $\mathcal G_{n^2}(Z_1)\leq nm \left( e^{-\frac{\rho^2 (n-m+1)}{4}}+ e^{-\frac{\rho^2 (n-m+1)}{16}} \right)$, 
where we have applied a union bound argument in $i$.

For every $1\leq j \leq m$, we have \[\left(\sqrt{n} -(1-\rho)^{\frac{1}{2}}\sqrt{n-j+1}\right)^2\leq \left(\sqrt{n} -(1-\rho)^{\frac{1}{2}}\sqrt{n-m}\right)^2.\]
Using this  and  Proposition \ref{concentraciondenormas} as in the proof of Proposition \ref{cotasH}, we get that 
\begin{align*} 
\|H_i^m\|^2&\leq   \left(\sqrt{n} -(1-\rho)^{\frac{1}{2}}\sqrt{n-m}\right)^2 \sum_{j=1}^m\nu_{j,i}^2 \\
&\leq \frac{1}{(1-\rho')^2}\frac{m}{n}\left(\sqrt{n} -(1-\rho)^{\frac{1}{2}}\sqrt{n-m}\right)^2,
\end{align*} 
outside of the sets $Z_1$ and $Z'_1$, where $Z_1$ is as above and  $\mathcal G_{n^2}(Z_1')\leq n e^{-\frac{{\rho'}^2 m}{4}}.$ 

We multiply and divide by the conjugate and lower bound the denominator by $n$, and we get

\[\sup_i \|H_i^m\|^2 \leq \frac{1}{(1-\rho')^2} \frac{m^2}{n} \left(\rho^2\frac n m +2\rho +\frac m n\right).\]

If we choose now, for instance, $\rho=\frac{m^\frac{1}{3}}{n^{\frac{1}{2}}}$, $\rho'=m^{-\frac{1}{3}}$ and  using that $\frac m n \rightarrow 0$, we have that
\[\sup_i\|H_i^m\|^2=o\left(\frac {m^2} n\right)\]
with probability exponentially small in $m$. 
Thus, in this regime $\|H_i^m\|$ is negligeable compared to $\|G_i^m\|$. 

\smallskip

Since \[\|G_i^m\|-\|H_i^m\|\leq \|F_i^m\|\leq \|G_i^m\|+ \|H_i^m\|, \] this finishes the proof. 
\end{proof}

\section{The supremum norm}\label{Jiangsthingy}
In this section we use Theorem \ref{main} to prove Theorem \ref{Jiangnuestro}, that describes the asymptotic probabilistic behaviour of $\epsilon_n(m)=\sup_{1\leq i \leq n, 1\leq j \leq m} |y_{i,j}- \sqrt{n}u_{i,j}|$.

If we choose $m_n=\frac{\beta n}{\log n}$ or  $m_n=o\left(\frac{n}{\log n}\right)$ we immediately obtain Corollary \ref{Jiangrefinado}. This result should be compared with \cite[Theorem 3]{Ji}: In it, Jiang showed that if  $Y_n$ is a Gaussian random matrix and $U_n$ is the result of its Gram-Schmidt orthonormalization, then $\epsilon_n(m)$ converges to 0 in probability if and only if $m_n=o\left(\frac{n}{\log n}\right)$, and he also showed that if $m_n=\frac{\beta n}{\log n}$ then $\epsilon_n(m)$ converges in probability to $2\sqrt{\beta}$. 
Our Corollary \ref{Jiangrefinado} shows the existence of couplings between a Gaussian matrix $Y_n$ and a Haar distributed orthogonal matrix $U_n$ such that $\epsilon_n(m)$ also converges to 0 in probability if and only if $m_n=o\left(\frac{n}{\log n}\right)$ but now, when  $m_n=\frac{\beta n}{\log n}$, the upper bound for $\epsilon_n(m)$ converges in probability to  $\sqrt{2\beta}$.

Before we start our reasonings, we state and prove for completeness a lemma which is well known, but for which we have not found an explicit reference.
\begin{lemma}\label{cotaunitaria} Let ${\{\bf w_j}=(w_{ij})_{i=1}^n\}_{j=1}^m$ be m unitary vectors each of which is uniformly distributed in the sphere of $\mathbb R^n$. Then, for any $\epsilon>0$ we have $$Pr\left(\sup_{i,j}|w_{i,j}|> (1+\epsilon)\frac{\sqrt{2\log (nm)}}{\sqrt{n}}\right) \stackrel{n} \longrightarrow 0,$$
independently of the number of vectors $m$. On the other hand, if $m\leq \alpha n$ for some fixed number $\alpha$ we have
 $$Pr\left( \sup_{i=1,...,n} |w_{i,j}|\geq(1-\epsilon)\frac{\sqrt {2\log n}}{\sqrt{n}} \text{   }\text{  for all $j\in\{1,...,m\}$} \right)\stackrel{n}\longrightarrow 1.$$ 
\end{lemma}
\begin{proof}
In order to prove the first expression, we consider the function that projects a unitary vector in $\mathbb R^n$ onto its $i$th coordinate. This function has Lipschitz constant $1$ and its median is $0$. Thus, a straightforward consequence of Levy's lemma \cite{MiSc} shows that, for $1\leq i\leq n$ and $1\leq j\leq m$,
$$Pr\left(|w_{i,j}|>t\right)\leq \sqrt{\frac \pi 2 } e^{-(n-1)t^2/2}.$$
Taking $t=(1+\epsilon)\frac{\sqrt{2\log (nm)}}{\sqrt{n}}$ and applying a union bound we get
$$ Pr\left(\sup_{i,j} |w_{i,j}|>(1+\epsilon)\frac{\sqrt{2\log (nm)}}{\sqrt{n}}\right)\leq \sqrt{\frac \pi 2 } (nm)^{-2\epsilon-\epsilon^2 + \frac {(1+\epsilon)^2}{n}}\stackrel{n} \longrightarrow 0.$$
For the second part of the statement we first consider an arbitrary Gaussian vector ${\bf g}=(g_{1},...,g_{n}) \in \mathbb R^n$. It follows from the independence of the coordinates of ${\bf g}$  and the lower bound in Lemma \ref{cotaGaussiana} that, for any $t>0$, 
\begin{equation}\label{aux0}Pr\left( \sup_{i=1,...,n} |g_{i}|<t\right) = \left(Pr\left( |g_{1}|<t\right)\right)^n \leq \left(1- \frac 2 {\sqrt {2\pi}} \frac {t} {1+t^2}e^{-t^2/2}\right)^n.
\end{equation}

Using the fact that a uniform unitary vector distributes like a normalized Gaussian vector together with a union bound, we deduce that for every $t,s>0$ we have
\begin{align*}
&Pr\left(\exists j\in\{1,...,m\}: \sup_{i} |w_{i,j}|<t\right)
\\&\leq m Pr\left( \sup_{i} |w_{i,1}|<t\right)=m Pr\left(  \frac{\sup_{i}|g_{i}|}{\|{\bf g}\|}<t\right)\\
&\leq m Pr\left( \big\{\sup_{i}|g_{i}|<ts \big\} \cap \big\{ \|{\bf g}\|\leq s\big\} \right) + m Pr\left( \|{\bf g}\|>s\right)\\&\leq m Pr\left(\sup_{i}|g_{i}|<ts \right) +m Pr\left(\|{\bf g}\|>s\right).
\end{align*}

We fix $t=(1-\epsilon)\frac{\sqrt{2\log n}}{\sqrt{n}}$, $s=\frac {\sqrt{n}}{\sqrt{1-\delta}}$ and we apply Proposition \ref{concentracion} to get that, for $0<\delta<1$, $0<\epsilon<1$, with $\frac{(1-\epsilon)}{\sqrt{1-\delta}}<1$,
\begin{align*}
&Pr\left(\exists j\in\{1,...,m\} : \sup_{i} |w_{i,j}|<(1-\epsilon)\frac{\sqrt{2\log n}}{\sqrt{n}}\right)\\& \leq m Pr\left( \sup_{i}  |g_{i}|<\frac{(1-\epsilon)}{\sqrt{1-\delta}}\sqrt {2\log n}\right)+m e^{-\delta^2n/4}.
\end{align*}

Using Equation (\ref{aux0}) we easily conclude that
$$Pr\left( \exists j\in\{1,...,m\} : \sup_{i=1,...,n} |w_{i,j}| <(1-\rho)\frac{\sqrt {2\log n}}{\sqrt n}\right) \stackrel{n}\longrightarrow 0,$$
provided that $m$ grows as in the statement. Considering the complementary event we obtain the desired result. 
\end{proof}%

Theorem \ref{main} gives us control over the Euclidean norm of the rows $F_i^m$ of $Y_n-\sqrt{n}U_n$. We will use these estimates to obtain information about the supremum $$\epsilon_n(m)=\sup_{1\leq i \leq n, 1\leq j \leq m} |y_{i,j}-\sqrt{n}u_{i,j}|.$$

First of all we note that $U_n$ is the Gram-Schmidt orthonormalization of $Y_n$. Therefore, the columns of  $Y_n-\sqrt{n}U_n$ are not equally distributed. For instance, it is very easy to see that, with very high probability, their Euclidean norm is strictly increasing. In turn, this implies that the coordinates of each of the $F_i^m$ are not equally distributed. To avoid this problem, we will randomly choose a slightly better coupling than the one given by the Gram-Schmidt orthonormalization procedure. 
\begin{proof}[\bf Proof of Theorem \ref{Jiangnuestro}:]
Let $Y_n, U_n$ be as in Theorem \ref{main}. We consider a Haar distributed orthogonal random matrix  $V_m\in \mathcal O(m)$. We define the orthogonal matrix $V=(v_{i,j})_{i,j=1}^n \in \mathcal O(n)$ by 
$$V=
 \begin{pmatrix}
  V_m & 0  \\
  0 & I_{n-m}
 \end{pmatrix}.
$$

We now define $Y'=YV$, $U'=UV$. Due to the orthogonal invariance of both the Gaussian distribution and the Haar distribution, we have that $Y'$ is a random Gaussian matrix and $U'$ is Haar distributed in the orthogonal group. Note that $U'$ is {\em not} the Gram-Schmidt orthonormalization of $Y'$. 

We have now that $Y_n'-\sqrt{n}U_n'=(Y_n-\sqrt{n}U_n) V$. Call $F_{i,j}$ to the $j^{th}$ coordinate of the vector $F_i^m$ defined as in Theorem \ref{main}. Then the first $m$ coordinates of the $i^{th}$ row of  $Y_n'-\sqrt{n}U_n'$ form the vector ${\bf x}_i=(x_{i,j})_{j=1}^m$, where $$x_{i,j}=\sum_{k=1}^m F_{i,k} v_{k,j}.$$

Therefore ${\bf x}_i\in \mathbb R^m$ is a vector whose direction is uniformly distributed and it verifies $\|{\bf x}_i\|=\|F_i^m\|$. That is, for every $1\leq i \leq n$, $\frac {{\bf x}_i} {\|F_i^m\|}$ is a unitary vector uniformly distributed.

We will first prove the upper bound of $\epsilon_n(m)$. It follows from the first part of  Lemma \ref{cotaunitaria} that, for every $t>0$,  $$Pr\left(\sup_{i,j} \frac {|x_{i,j}|}{\|F_i^m\|}> (1+\varepsilon)\frac{\sqrt{2\log (nm)}}{\sqrt{m}}\right)\stackrel {n}\longrightarrow 0.$$
We have that
\begin{equation}\label{supremumsup}
Pr\left( \epsilon_n(m)>(1+\varepsilon)\sup_i\|F_i^m\|\frac{\sqrt{2\log (nm)}}{\sqrt{m}} \right)\stackrel {n}\longrightarrow 0.
\end{equation}

We recall that, according to Theorem \ref{main}, there exists $0<\delta<\frac{1}{2}$ such that 
$$\sup_i\|F_i^m\| \leq \sqrt{  \left(2- \frac{4}{3} \frac{(1-(1 -\alpha)^{3/2})}{\alpha}\right)m} + O(m^{\frac{1}{2}-\delta}) := \sqrt{\varphi(\alpha) m }+ O(m^{\frac{1}{2}-\delta})$$
with probability exponentially close to 1. Putting this together with Equation (\ref{supremumsup}) we get the upper bound.

For the lower bound, we will consider the column vectors ${\bf \tilde x}_j=(x_{i,j})_{i=1}^n$, with $j=1,\cdots, m$, corresponding to the matrix $Y_n'-\sqrt{n}U_n'$. We first note that the normalized vectors $\frac{{\bf \tilde x}_j}{\|{\bf \tilde x}_j\|}$ can be assumed to be uniformly distributed for every $j$. Indeed, this follows from the fact that the initial coupling defined via the Schmidt orthonormalization procedure is invariant under left multiplication by an orthogonal matrix.  Since left and right multiplication are commutative operations, the invariance under left multiplication by an orthogonal matrix also holds for the random matrix $Y_n'-\sqrt{n}U_n'$. Then, according to Lemma \ref{cotaunitaria} we have
\begin{align*}
Pr\left( \sup_{i=1,...,n} |x_{i,j}|\geq(1-\epsilon)\|{\bf \tilde x}_j\|\frac{\sqrt {2\log n}}{\sqrt{n}} \text{   }\text{  for all $j\in\{1,...,m\}$} \right)\stackrel{n}\longrightarrow 1.
\end{align*}

Now, we will use the trivial fact that $\sum_{j=1}^m\|{\bf \tilde x}_j\|^2=\sum_{i=1}^n\|{\bf x}_i\|^2=\sum_{i=1}^n\|F_i^m\|^2$ and the lower bound provided in Theorem \ref{main} to conclude that $\sum_{j=1}^m\|{\bf \tilde x}_j\|^2\geq n\big(\sqrt{\varphi(\alpha)m}-O(m^{\delta})\big)^2$, so 
\begin{align*}
\sup_{j=1,\cdots, m} \|{\bf \tilde x}_j\|\geq  \sqrt{n}\big(\sqrt{\varphi(\alpha)}-O(m^{\delta-\frac{1}{2}})\big),
\end{align*} with probability exponentially close to 1 as $n$ grows to infinity. Putting the two previous estimates together we easily deduce that 
\begin{equation*}
Pr\left(\epsilon_n(m)\geq(1-\varepsilon) \left(\sqrt{\varphi(\alpha) }- O(m^{\delta-\frac{1}{2}})\right) \sqrt {2\log n}\right)\stackrel{n}\longrightarrow 1,
\end{equation*}
which finishes the proof.
\end{proof}
\begin{remark}
We expect the lower bound of Theorem \ref{Jiangnuestro} to be  $\sqrt{\varphi(\alpha)} \sqrt{2\log (nm)}$. To prove that, one needs to overcome the lack of independence of the rows $F_i^m$. 
\end{remark}

In our previous results, $\alpha=\frac{m}{n}$ was a constant number. One can easily check that the proof of Theorem \ref{Jiangnuestro} works in a more general context. In particular, the same argument works if one considers $m_n=[\beta \frac{n}{\ln n}]$ for any constant $\beta>0$. Note that in this case $\alpha_n=\frac{\beta}{\ln n}\stackrel{n}\rightarrow 0$. With this at hand, we can state and prove  the announced improvement of \cite[Theorem 3]{Ji}.
\begin{corollary}\label{Jiangrefinado}
For each $n \geq  2$, there exist matrices $Y'_n = (y'_{ij})_{i,j=1}^n$ and 
$U'_n = (u'_{i,j})_{i,j=1}^n$ whose $2n^2$ entries are real random variables defined on the same probability space such that
\begin{itemize}
\item[(i)] the law of $U'_n$ is the normalized Haar measure on the orthogonal  group $\mathcal O(n)$; 

\item[(ii)] $\{ y'_{i,j}; 1\leq i, j\leq n\}$ are independent standard normals;
\item[(iii)] set $$\epsilon_n(m)=\max_{1\leq i \leq  n, \, 1 \leq j \leq m}  |\sqrt{n} u'_{i,j}-y'_{i,j}|$$  for  $m = 1, 2, \cdots, n$. Then, for any constant $\beta>0$ and $m_n=[\frac{\beta n}{\log n}]$, for any $\varepsilon>0$, $\epsilon_n(m)$ belongs to the interval $$(\sqrt{\beta}-\varepsilon, \sqrt{2\beta}+ \varepsilon)$$ with probability $1-o(1)$. Moreover, if we make $m_n=o(\frac{n}{\ln n})$, then we have that $\epsilon_n(m)\rightarrow 0$ in probability as $n\rightarrow \infty$. 
\end{itemize}

\end{corollary}
\begin{proof} Let $Y'_n$, $U'_n$ be as in Theorem \ref{Jiangnuestro}. As previously discussed, Theorem \ref{Jiangnuestro} also holds for $m_n=[ \frac{\beta n}{\ln n}]$ with $\beta>0$ and in this case $\alpha_n=\frac{\beta}{\ln n}$. Then, there exists $0<\delta<\frac{1}{2}$ such that for any $\varepsilon',\varepsilon''>0$ and for $n$ large enough we have 
$$\epsilon_n(m)\leq(1+\varepsilon') \left(\sqrt{\frac{\alpha_n }{2}+ \alpha_n^2} + O(m_n^{-\delta})\right) \sqrt {2\log (nm_n)}\leq (1+\varepsilon')(\sqrt{2 \beta} + \varepsilon'')$$
with probability tending to $1$. Here, we have used that for any $0<\alpha\leq1$ the function $\varphi(\alpha)=\left(2- \frac{4}{3} \frac{(1-(1 -\alpha)^{3/2})}{\alpha}\right)$ is upper bounded by $\frac  \alpha 2 +\alpha^2$. This proves the upper bound in the first part of the $[\frac{\beta n}{\log n}]$ statement. For the lower bound we reason similarly using the lower bound in Theorem \ref{Jiangnuestro}. The statement about $m_n=o(\frac{n}{\ln n})$ follows straightforward by considering a sequence $b_n\rightarrow 0$ and the previous estimate.
\end{proof}
\section*{Acknowledgments}

We would like to thank the referee for many useful comments which helped to improve the readability of this paper and for suggesting us the explicit inclusion of Corollary 1.2.

\begin{sloppypar}
Author's research was supported by Spanish research projects MINECO (grants MTM2011-26912 and MTM2014-54240-P), Comunidad de Madrid (grant QUITEMAD+-CM, ref. S2013/ICE-2801), ``Ram\'on y Cajal" program and \mbox{ICMAT} Severo Ochoa project SEV-2015-0554 (MINECO), and the European CHIST-ERA project CQC (funded partially by MINECO grant PRI-PIMCHI-2011-1071).
\end{sloppypar}
\vskip 0.5 cm
\hfill \noindent \textbf{C. E. Gonz\'alez-Guill\'en} \\
\null \hfill Instituto de Matem\'atica Interdisciplinar, IMI\\
\null \hfill Departamento de Matem\'aticas del \\ \null \hfill
\'Area Industrial, E.T.S.I. Industriales, UPM \\ \null \hfill 
28006 Madrid, Spain
\\ \null \hfill\texttt{carlos.gguillen@upm.es}
\vskip 0.5cm
\hfill \noindent \textbf{Carlos Palazuelos} \\
\null \hfill Instituto de Ciencias Matem\'aticas, ICMAT\\
\null \hfill Facultad de Ciencias Matem\'aticas\\ \null \hfill
Universidad Complutense de Madrid \\ \null \hfill Plaza de Ciencias s/n.
28040, Madrid. Spain
\\ \null \hfill\texttt{carlospalazuelos@ucm.es}
\vskip 0.5cm
\hfill \noindent \textbf{I. Villanueva} \\
\null \hfill Instituto de Matem\'atica Interdisciplinar, IMI\\
\null \hfill Facultad de Ciencias Matem\'aticas\\ \null \hfill
Universidad Complutense de Madrid \\ \null \hfill Plaza de Ciencias s/n.
28040, Madrid. Spain
\\ \null \hfill\texttt{ignaciov@ucm.es}

\end{document}